%% file: 1.tex
\documentclass[12pt,a4paper,oneside]{amsart}

\usepackage{rotating}
\usepackage[usenames,dvipsnames]{pstricks}
\usepackage{epsfig}

\usepackage[section]{placeins}
\usepackage{perpage}
\MakePerPage[2]{footnote}
\newtheorem{theorem}{Theorem}[section]
\newtheorem{lemma}[theorem]{Lemma}
\newtheorem{claim}[theorem]{Claim}
\newtheorem{prop}[theorem]{Proposition}
\newtheorem{cor}[theorem]{Corollary}
\theoremstyle{definition}
\newtheorem{definition}[theorem]{Definition}

\theoremstyle{remark}
\newtheorem{remark}[theorem]{Remark}

\numberwithin{equation}{section}

\begin{document}

\title[An upper bound on the number of zeros of a spline]{An upper bound on the number of zeros of a piecewise polinomial function}
%    author one information
\author{Marco B. Caminati}
\address{Dipartimento di Matematica "Guido Castelnuovo" \\ Sapienza - Universit\`a di Roma}
\curraddr{}
\email{caminati@mat.uniroma1.it}
\urladdr{http://www.mat.uniroma1.it/$\sim$caminati/}
\thanks{}

	
\date{2008/03/16}

\dedicatory{}

\begin{abstract}
A precise tie between a univariate spline's knots and its zeros abundance and dissemination is formulated. \\
As an application, a conjecture formulated by De Concini and Procesi is shown to be true in the special univariate, unimodular case. \\
As a supplement, the same conjecture is shown, through computing a counterexample, to be false when unimodularity hypothesis is dropped.
\end{abstract}

\maketitle
\section{Introduction}
A piecewise polynomial function, or univariate spline, here just \emph{spline}, of degree $m$ is a function $s \in C^{m-1}(\mathbb{R},\mathbb{R})$  which admits a partition of the real axis into a finite or numerable family of bound intervals such that on each of them $s$ is the restriction of a polynomial of maximum degree $m$. The bounds of the intervals are called \emph{knots}. More strongly, the knots of $s$ are the points in which $s$ is not $C^{\infty}$, and each interval between consecutive knots is a \emph{polynomiality domain}.
If the set of knots coincides with $\mathbb{Z}$ one speaks of cardinal splines, a variation handy for theoretical and formal elaborations. We refer the reader to \cite{Schoenberg} for a more extensive treatment of splines. 
\\
One could think of $s$ as the adjoining of a finite or numerable set of polynomials done so as to keep up the best possible smoothness over the junction points we just defined as knots. This endeavour must bring some costraints on the choice of polynomials, and particularly on the zeros of $s$. Our goal is to elaborate on that, and draw an adaptation of the fundamental theorem of algebra which fully takes into account the cited constraints.\\
Firstly, we will restrict ourselves to the finite knots case. There's obviously no loss of generality in doing that, since one can always think a finite knots spline as the restriction on a bound interval of a numerably-knotted spline: what happens outside the outermost pair of nodes is of no interest to us at the moment. Once done that, one could try to apply the fundamental theorem of algebra to each of the polynomiality domain to get a first gross estimation of the upper bound over the number of zeros of $s$, which is our goal. We must hold a moment and realize a couple of thing before doing that: one, that the constraints we invoked will allow a much better approach to that, and two, that there can be polynomiality domains on which $s$ is identically zero, a case which in the case of polynomials does not pose much problem, and which here must be taken care of.
\\
So, we will specialize the definition of zeros to the one of \emph{separated zeros} (see \ref{ZeriSeparati}), which is necessary to work around the annoying event of a piecewise constant spline. Such a case for $m=1$, suitable to depiction (for the knots are identifiable at a glance with the ``spikes'') is portrayed in figure \ref{fig_zeriseparati}. 
\begin{figure}
\psscaleboxto(\textwidth,0)
{
\input{./fig01.tex}
}
\caption[]{}
\label{fig_zeriseparati}
\end{figure}
In its right part the result of inserting a flat zone between two adjacent polynomiality domains of its left part is shown. For our purpose of counting the zeros this insertion is not to be taken into account, and that's the motivation for the new definition.

\section{Main results}
\begin{definition}
\label{ZeriSeparati}
\mbox{ } 
\\
\begin{itemize}
\item
$a,b \in \mathbb{R}$ are said to be \emph{$f$-separated} iff $f$ is a function defined and non-constant on $\left[a,b\right]$, i.e. 
\begin{equation*}
\exists c \in \left[a,b\right]  /  f(c)\neq f(a)
\end{equation*}
\item
Broadening the above definition: a countable subset of $\mathbb{R}$  $\left\{ \ldots < a_{i}<
a_{i+1} < \ldots \right\}_{i \in \mathbb{Z}}$ is said to be \emph{$f$-separated} iff
\begin{gather*}
\forall i \in \mathbb{Z}, a_{i} \text{ and } a_{i+1} \text{ are $f$-separated.}      
\\
\Updownarrow
\\ 
\forall i \neq j \in \mathbb{Z}, a_{i} \text{ and } a_{j} \text{ are $f$-separated.}
\end{gather*} 
\end{itemize}
\end{definition}
When there will be no possibility of confusion we shall dispense ourselves from writing the prefix $f$\nolinebreak-. in front of ``separated''.
\begin{remark}
\label{tantoalchilo}
We can now draw our first gross estimation of an upper bound on the number of \emph{separated} zeros of $s$ as already sketched in the introduction.
It's plainly given by $m(N+1)$, where $N$ is the number of knots, and is obtained by applying the fundamental theorem of algebra to each polynomiality domain and taking advantage of the definition of separated zeros.
\end{remark}
This remark assures that the following definition is correct.
\begin{definition}
$Z(s) \in \mathbb{N}$ is the number of zeros of the finite-knotted spline $s$.
\end{definition}
\ref{fondamentale} can be viewed as the basic result of this paper from the technical point of view, expressing the ties we were seeking between the number of knots, the degree and the number of separated zeros of a spline having compact support. It's preceded by a simple analytic lemma.

\begin{lemma}[strong Rolle]
\label{Rolle}
Hypothesis:
\\
\begin{enumerate}
\item
$f$ is derivable in $\left[a,b\right]$ , $(a<b)$
\item
$f(a)=f(b)$
\item
$a, b$ are $f$-separated
\end{enumerate}
Thesis:
\\
$\exists z \in ]a,b[$ such that:
\begin{enumerate}
\item
${f}^\prime(z)=0$
\item
$a,z,b$ are $f'$-separated
\end{enumerate}
\end{lemma}
\begin{proof}
$f$, being continuous on $\left[a,b\right]$, admits maximum and minimum on $\left[a,b\right]$. 
Employing the separation hypothesis we can suppose, for example, the existence of
$z\in \left[a,b\right]/f \left(z \right)= \underset{ \left[ a,b \right] }{max}f>f \left( a \right) =f \left( b \right) $.
\\
Being $f$ derivable, $f ^\prime(z)=0$.

Moreover, $f ^\prime$ can't be constant on $\left[a,z\right]$, otherwise $f ^\prime$ would be there identically zero and $f \left(z \right)=f \left(a \right)$ against what just stated. Then $a$ and $z$ are $f^ \prime$-separated. Similarly so are $z$ and $b$. 
\end{proof}

Now our informal idea is to apply repeated derivation to lower the degree of the spline, and employ \ref{Rolle} to control the number of separated zeros thus elicited. We thus end up with a spline of degree one (a polygonal path), for which the connection between zeros and knots number is trivial and is formally expressed in \ref{sega}.
\\
At first we proceed with our idea on a particular kind of spline, having as outermost knots a couple of zeros of maximum order, which gives them the quality of being conserved through derivation. This is a quality because simplifies our reasoning and permits to deduce our first result.
This special kind of splines, subject of \ref{fondamentale}, encompasses, notably, the ones with compact support, but it's wider, containing all the splines which can be smoothly flattened outside a compact interval.
\begin{prop}
\label{fondamentale}
Be given a spline $s:\mathbb{R} \longrightarrow \mathbb{R}$ of degree $m$, being
$\left\{ \alpha_{0}, \ldots , \alpha_{n} \right\}$
its ordered knots ($m,n \in \mathbb{Z}^{+}$). 
\\
What's more, suppose that
$\alpha_{0}$ and $\alpha_{n}$ are \emph{separated zeros of order\footnote{
Here we mean order of a zero $z$ according the following analytic definition, where $f \in C^{m-1}$:
\begin{equation*}
\text{ord}(z,f):=min\left( \{j \in \{0, \ldots ,m-1 \} / f^{\left( j \right)}\left(z\right)\neq 0 \} \cup \{m\}\right)
\end{equation*}
Obviously this reduces to the algebraic definition given in terms of multiplicity of a zero in case $f$ is a polynomial.
}
$\geq m$}\footnote{This implies the order is either $m$ or $+ \infty$}

Then
\begin{itemize}
\item
$n \geq m+1$
\item
$s$ has at most $n-m-1$ separated zeros in $\left] \alpha_{0}, \alpha_{n} \right[$.
\end{itemize}
\end{prop}
\begin{proof}
It suffices to show that
\begin{eqnarray}
Z(s) \leq n-m+1&,
\end{eqnarray}
which immediately implies the second statement, but also the first, because 
 (cfr. \ref{tantoalchilo}) $Z$ is well defined and non negative, so in this case we get $Z(s) \geq 2$.
\\
By induction on $m$.
$m=1$ means that $s$ gives a polygonal path composed of $n$ line segments; this path can't, quite obviously,
admit more than $n$ separated zeros (by the way, that's true even dropping the last request on $\alpha_{0}$ and $\alpha_{n}$). For a formal proof of this last elementary assertion see \ref{sega}.
\\
Now let $m > 1$. Suppose we found $k+2$ separated zeros for $s$; we denote them the following way:
\begin{eqnarray}
\alpha_{0}=: z_{0} < z_{1} < \ldots < z_{k} < z_{k+1} := \alpha_{n}
\end{eqnarray}

Now we apply \ref{Rolle} (explicitly noting that we are entitled to do so, being $s$ derivable, even more being it in $C^{m-1}$) to each $\left[z_{i-1},z_{i} \right], i=1,2,\ldots,k+1$ and work out
\begin{eqnarray}
z_1^{\prime}<\ldots<z_k^{\prime}<z_{k+1}^{\prime} & \left(z^{\prime}_{i} \in \left] z_{i-1}, z_{i}\right[\right)
\end{eqnarray}
$k+1$ separated zeros for $s^{\prime}$.\\
Now, again, $s^{\prime}$ is a spline of degree $m-1$ and its knots are the knots of $s$
Moreover,  $\alpha_{0}$ and $\alpha_{n}$ are zeros of order $\geq m-1$ for $s^{\prime}$, we can conclude from the inductive hypothesis that
\begin{eqnarray}
k+1\leq n-(m-1) -1 \Leftrightarrow k \leq n-m-1
\end{eqnarray}
\end{proof}
The two statements proved in \ref{fondamentale} give respectively the following couple of corollaries.
\begin{cor}
\label{supportominimo}
The compact support of a spline of degree $m \geq 1$ contains at least $m+1$ consecutive polynomiality domains.
\\
Equivalently\footnote{This second formulation is suited for splines with countably many nodes as well.}: each connected component of the support of $s$ has ``length'' at least $m+1$ in term of polynomiality domains.
\end{cor}
\begin{proof}
Just note that the nodes delimiting the support are forcibly also zeros of order $m$, and that they are obviously mutually separated.
\\
The second formulation is obtained by choosing a connected component of the support and smoothly zeroing the spline outside of it, then applying what just showed.
\end{proof}

\begin{cor}
A smallest-support spline (see \ref{supportominimo}) never attains zero value inside the support itself. In particular, it's either always strictly positive or negative there.
\end{cor}

Now we wish to write the general spline as a linear combination of translations of compact-supported splines, so as to employ \ref{fondamentale} and desume the relations between $m,n$ and the number of knots we are looking for. In other words, given a generic spline with finitely many knots, we would like to extend it outside its outermost knots to a compact-supported spline, which is subject to \ref{fondamentale}.
To this end we introduce a simple technical lemma. 
\begin{lemma}
\label{estensione}
Given a spline $s$ of degree $m$ and knots $\alpha_{0}<\ldots<\alpha_{n}$, there is a compact-supported spline $\bar{s}$ extending $s|_{\left[\alpha_{0},\alpha_{n}\right]}$ and such that the set of its knots is included in
\begin{align*}
\left\{
\alpha_{0}-m,\alpha_{0}-m+1,\ldots,\alpha_{0},\ldots,\alpha_{n},\alpha_{n}+1,\ldots,\alpha_{n}+m
\right\}
\end{align*}
\end{lemma}
\begin{proof}
We need to introduce $B_{m}$, the $B$-spline of degree $m$\footnote{To avoid confusion the reader is warned that in \cite{Schoenberg} $M_{m}$ denotes the B-spline of degree $m-1$}. It's a particularly well-behaved cardinal spline supported on $\left[ 0, m+1 \right]$ (cfr. \ref{supportominimo}) whose prominent property, and the one we will need, is that translating it at steps of length one we obtain a basis of the whole space of cardinal splines. Again, refer to \cite{Schoenberg} for full treatment.
We can limit ourselves to show the thesis on the left side, and there's no loss assuming $\alpha_{0}=0$. 
Consider\protect\footnote{$T_{\lambda}$ denotes the unidimensional translation operator: $T_{\lambda} f:x \mapsto f(x- \lambda )$}
\begin{align*}
s^{-}(x):=\sum_{j=-m}^{0}\lambda_{j}T_{j}B_{m}(x).
\end{align*}
As just said, we can find a set of coefficients $\left\{ \lambda_{j} \right\}_{j=-m,\ldots ,0}$ such that $s^{-}$ and $s$ coincide on $\left[0,min\{\alpha_{1},1\}\right]$. This implies that the ``gluing''
\begin{gather*}
\left(s^{-}\left|\right|s\right)(x):=\left\{
\begin{matrix}
s^{-}(x), & \text{se } x \leq 0
\\
s(x), & \text{se } x \geq 0
\end{matrix}
\right.
\end{gather*}
is still a spline of degree $m$ having all of its knots included in the set
\begin{equation*}
\left\{
\alpha_{0}-m,\alpha_{0}-m+1,\ldots,\alpha_{0},\ldots,\alpha_{n}
\right\}
\end{equation*}
and is obviously identically zero at the left of $\alpha_{0}-m$.
\end{proof}
We finally get to the main result of this paper.
\begin{theorem}
\label{principale}
A spline of degree $m$ and knots $\alpha_{0}<\ldots<\alpha_{n}$ has at most $n+m-1$ separated zeros in $\left[\alpha_{0},\alpha_{n}\right]$. 
\end{theorem}
\begin{proof}
Called $s$ the given spline, let us build, as granted by \ref{estensione}, its compact-supported extension $\bar{s}$. 
The latter has at most $N=n+1+2m$ knots, so by \eqref{fondamentale} will also have at most $N-m-2=n+m-1$ separated zeros in   $\left]\alpha_{0}-m,\alpha_{n}+m\right[$.
So its zeros in $\left[\alpha_{0},\alpha_{n}\right]$ are at most $n+m-1$ as well, whence the thesis since there $s$ and $\bar{s}$ coincide by construction.
\end{proof}
The following corollary give a precise meaning to the idea that \ref{principale} permits to render a spline zero everywhere by scattering a sufficient number of zeros inside its nodes, and in the right places.
This way we get a result not resorting to the notion of separated zeros.
\begin{cor}
\label{finale}
If $s$ is a spline of degree $m$ and knots $\alpha_{0} <  \ldots < \alpha_{n}$ such that
\begin{enumerate}
\item
$
\label{ipotesi1}
\text{card}(s^{-1}(\{0\}) \cap \left[ \alpha_{0}, \alpha_{n} \right]) \geq n+m
$
\item
\label{ipotesi2}
$
 \forall j=1,\ldots,n: \quad \left( \left| s \left( \alpha_{j-1} \right) \right| + \left| s \left(\alpha_{j}\right) \right| \right) \cdot s\left(\left] \alpha_{j-1},\alpha_{j}\right[ \right) \supset \{0\}
$ 
\end{enumerate}
Then $s$ is identically zero on $\left[ \alpha_{0}, \alpha_{n} \right]$.
\end{cor}
The formulae \eqref{ipotesi1} and \eqref{ipotesi2} are a compact symbolic way to render the request that $s$
is zero in at least $n+m$ points (hypothesis \eqref{ipotesi1}) which moreover are ``scattered enough'' along $\left[\alpha_{0}, \alpha_{n}\right]$ (hypothesis \eqref{ipotesi2}):
indeed \eqref{ipotesi2} succinctly expresses the condition that $s$ is zero either in \emph{one} point of the interior or in \emph{both} the ends of any given polynomiality domain.
\begin{proof}
The $m+n$ zeros, can't be separated as by \ref{principale}, therefore $s$ is identically zero on at least one polynomiality domain.
We shall see that it is zero on the \emph{whole} $\left[ \alpha_{0}, \alpha_{n} \right]$.
\\
By contradiction.
There would be, for example,
\begin{equation*}
j \in \left\{ 1,2,\ldots,n \right\}/
\left\{
\begin{matrix}
p:=s|_{\left[ \alpha_{j-1},\alpha_{j} \right] } \neq 0
\\
s|_{ \left[ \alpha_{j},\alpha_{j+1} \right] } =0
\end{matrix}
\right.
\end{equation*}
To preserve the regularity of $s$, $p$ should have a zero of order $m$ in $\alpha_{j}$.
\\
But according to \eqref{ipotesi2}, $p$ has an additional  zero in $\left[ \alpha_{j-1},\alpha_{j} \right[$, whilst its degree is just $m$.
\end{proof}
\ref{finale} is enough to prove the claim mentioned in the abstract, so we consider this section has reached its goal. We just give a formal proof of the highly intuitive statement obtained by restricting \ref{fondamentale} to the special case $m=1$, which in turn has been employed in the proof of \ref{fondamentale} itself.
\begin{lemma}
\label{sega}
A spline of degree $1$ and knots $\alpha_{0}<\ldots<\alpha_{n}, n \in \mathbb{Z^{+}}$ admits at most  $n$ separated zeros in $\left[ \alpha_{0},\alpha_{n}\right]$.
\end{lemma}
\begin{proof}
By induction on $n$. If $n=1$, and $\left[\alpha_{0},\alpha_{1}\right]$ contains two distinct zeros, they oviously cannot be separated.
Let us now suppose that $s$ is a spline of knots $\alpha_{0} < \ldots < \alpha_{n}$, and let $k$ be the maximum number of separated knots of $s$ restricted to $\left[ \alpha_{0},\alpha_{n-1}\right]$.
By the inductive hypothesis $k\leq n-1$. Now we reason distinguishing two separate cases:
\begin{itemize}
\item
$s(\alpha_{n-1})=0$
\\ 
In this case the zeros of $s$ in $\left] \alpha_{n-1},\alpha_{n}\right]$ can't be separated from $\alpha_{n-1}$, because $s|_{\left[ \alpha_{n-1},\alpha_{n} \right]}$ is a first degree polynomial.
\item
$s(\alpha_{n-1}) \neq 0$
\\
In this case $s|_{\left[\alpha_{n-1},\alpha_{n} \right]}$ is a non-identically-null polynomial, having therefore at most one zero, so  $s$ has at most one zero in $\left[ \alpha_{n-1},\alpha_{n}\right]$.
\end{itemize}
In any case the total number of separated zeros in $\left[ \alpha_{0},\alpha_{n} \right]$ is $\leq k+1 \leq n$.
\end{proof}

\section{Validating the conjecture}
We briefly introduce a framework for multivariate splines which generalizes the unidimensional case treated until now; see \cite{Procesi} for details.
In the unidimensional case we had polynomiality domains given by bounded intervals and the boundaries between them given by points (knots). In the case of $\mathbb{R}^{s}$ they become respectively polyhedra and hyperplanes. An interesting fact is that under unimodular hypothesis one can still build a family, given by the so-called Box splines, having many analogies with that of the B-splines.
One just has to slightly modify the way of describing the polynomiality domains to adequate it to the increased degrees of freedom. Instead of explicitly enumerate the knots, one gives a finite $m$-ple $X:= \left( a_{1}, \ldots, a_{m} \right)$ of vectors of $\mathbb{R}^{s}$, and builds the lattice it generates. The cells thus obtained are the polynomiality domains we wanted to describe.
Given $X$, a unique box spline $B_{X}$ is determined upon requests of completeness and minimum support. Its degree is $m-s$ and its support is\footnote{The notation $\sum_{A}^{B}$, where $V \supseteq A$ is a vector space over the field $\mathbb{K} \supseteq B$, denotes all the possible linear combinations of vectors in $A$ with coefficients in B.} the convex hull $\sum_{X}^{[0,1]}$ of $\sum_{X}^{\left\{ 0, 1 \right\}} $.
At this point we have the means to restate the conjecture stated in \cite{Procesi}, section entitled ``Explicit Projections''.
\begin{claim}
\label{Congettura}
Given a $m$-ple $X:=\left( a_{1}, \ldots, a_{m} \right) \subseteq \mathbb{R}^{s}$ and defined as
\begin{equation*}
 \{\omega_1,\omega_{2},\ldots,\omega_{n}\}=\Omega:=\overset{\circ}{\Sigma}\vphantom{\Sigma}_{X}^{[0,1]}\cap \frac{1}{2}\Sigma_{X}^{\mathbb{Z}}
\end{equation*}
the set of the semi-integral points of the interior\footnote{If $X$ does not span $\mathbb{R}^{s}$ one must resort to the notion of relative interior, cfr \cite{Procesi}.} of the support $\sum_{X}^{\left[ 0, 1 \right]}$ of $B_{X}$, the matrix
\begin{align*}
\left( A_{X} \right)_{i,j}:=B_X(\Sigma_X^{\{1\}}+\omega_i-2\omega_j)
\end{align*}
is invertible.
\end{claim}
\subsection{Not without unimodularity}
The symbolic calculation software \cite{Caminati} has been exploited to show that in case $X$ is not unimodular there exists a simple counterexample to \ref{Congettura}. A symbolic manipulation had to be utilized since the determinant to be computed happens to be very near to zero even in simple cases. The counterexample is obtained by taking $X=B_{2}:=((1,0),(1,1),(0,1),(-1,1))$ and getting
\begin{equation*}
A_{B_{2}}=
\input{./matrice_b2.tex}
\end{equation*}
which has symbolically-computed determinant $0$.
On the other hand, the \emph{unimodular} case $X=A_{2}:=((1,0),(1,1),(0,1))$ complies with \ref{Congettura}:
\begin{eqnarray*}
A_{A_2}=\input{./matrice_a2.tex} & det(A_{A_2})=\frac{1}{64}
\end{eqnarray*}
which leads to our last result.
\subsection{True in unidimensional, unimodular case}
In this case $X$ must be made up of (repeated) $1$'s and/or $-1$'s, so that we can reduce to the case in which
\begin{enumerate}
\item
$X=X_{m}:=\underbrace{(1,1,\ldots,1)}_{m+1\textrm{ times}}$
\item
$B_{X} = B_{m}$
\item
\begin{align*}
\left\{\frac{1}{2},1,\frac{3}{2},2,\ldots,m+\frac{1}{2}\right\} & , & \left| \Omega \right| = 2m+1
\end{align*}
\end{enumerate}
\begin{theorem}
$A_{X_{m}}$ is nonsingular.
\end{theorem}
\begin{proof}
Let us regard the $j$-th column $c_{j}$ of $A_{X_{m}}$ as the image through the translated B-spline
\begin{align*}
T_{2\omega_j - \Sigma_X^{\{1\}}}B_{X_{m}}=:T_{\lambda_j}B_{X_{m}}
\end{align*}
of the set
\begin{align*}
\Omega = \left\{\frac{1}{2},1,\frac{3}{2},2,\ldots,m+\frac{1}{2}\right\}
\end{align*}
Now let $\sum_{j=1}^{2m+1}d_j \cdot c_{j}$ be a null linear combination of the $c_{j}$'s. In our current view this means that the cardinal spline $\sum_{j=1}^{2m+1}d_j \cdot T_{\lambda_j}B_m$ is zero in $2m+1$ distinct points over $\left[ 0, m+1 \right]$  satisfying also request \eqref{ipotesi2} of \ref{finale}, which then allows to conclude that 
\begin{align*}
\left. \sum_{j=1}^{2m+1}d_j \cdot T_{\lambda_j}B_m \right|_{\left[ 0, m+1 \right]} \equiv 0
\end{align*}
But
\begin{align*}
T_{\lambda_1}B_m,\ldots,T_{\lambda_{2m+1}}B_m
\end{align*}
is a linear independent system and, moreover, neither of its elements is identically zero on $\left[ 0, m+1 \right]$, so the only possibility is that $d_{j} = 0 \quad \forall j=1, \ldots, 2m+1$.
\end{proof}

\end{document}

%% file: fig01.tex
%LaTeX with PSTricks extensions
%%Creator: 0.45.1
%%Please note this file requires PSTricks extensions

%\rotate{90}

%\psset{xunit=.5pt,yunit=.5pt,runit=.5pt}
\psset{xunit=1pt,yunit=1pt,runit=1pt}
\begin{pspicture}(1259.5,363.5)
{
\newrgbcolor{curcolor}{0 0 0}
\pscustom[linewidth=3.26097465,linecolor=curcolor]
{
\newpath
\moveto(11.2956346,173.236517)
\lineto(511.2979946,173.236517)
\lineto(511.2979946,173.236517)
}
}
{
\newrgbcolor{curcolor}{0 0 0}
\pscustom[linewidth=3.26097465,linecolor=curcolor]
{
\newpath
\moveto(510.843447,172.7283447)
\lineto(488.116067,185.43272845)
\lineto(488.116067,160.02396095)
\lineto(510.843447,172.7283447)
\closepath
}
}
{
\newrgbcolor{curcolor}{0 0 0}
\pscustom[linewidth=3.54330707,linecolor=curcolor,linestyle=dashed,dash=3.54330709 3.54330709]
{
\newpath
\moveto(175.74289,361.620132)
\lineto(175.74289,175.379887)
}
}
{
\newrgbcolor{curcolor}{0 0 0}
\pscustom[linewidth=3.54330707,linecolor=curcolor,linestyle=dashed,dash=3.54330709 3.54330709]
{
\newpath
\moveto(212.66191,321.869372)
\lineto(212.66191,173.130647)
}
}
{
\newrgbcolor{curcolor}{0 0 0}
\pscustom[linewidth=3.54330707,linecolor=curcolor,linestyle=dashed,dash=3.54330709 3.54330709]
{
\newpath
\moveto(135.66191,263.378477)
\lineto(135.66191,175.751417)
}
}
{
\newrgbcolor{curcolor}{0 0 0}
\pscustom[linewidth=3.54330707,linecolor=curcolor,linestyle=dashed,dash=3.54330709 3.54330709]
{
\newpath
\moveto(377.74586,353.922862)
\lineto(377.74586,175.072377)
}
}
{
\newrgbcolor{curcolor}{0 0 0}
\pscustom[linewidth=3.54330707,linecolor=curcolor,linestyle=dashed,dash=3.54330709 3.54330709]
{
\newpath
\moveto(269.49771,174.500007)
\lineto(269.49771,1.802687)
}
}
{
\newrgbcolor{curcolor}{0 0 0}
\pscustom[linewidth=3.54330707,linecolor=curcolor,linestyle=dashed,dash=3.54330709 3.54330709]
{
\newpath
\moveto(300.95155,174.500007)
\lineto(300.95155,63.315927)
}
}
{
\newrgbcolor{curcolor}{0 0 0}
\pscustom[linewidth=3.54330707,linecolor=curcolor,linestyle=dashed,dash=3.54330709 3.54330709]
{
\newpath
\moveto(318.2477,174.068337)
\lineto(318.2477,39.181677)
}
}
{
\newrgbcolor{curcolor}{0 0 0}
\pscustom[linewidth=3.54330707,linecolor=curcolor,linestyle=dashed,dash=3.54330709 3.54330709]
{
\newpath
\moveto(452.201545,293.167612)
\lineto(452.201545,174.878567)
}
}
{
\newrgbcolor{curcolor}{0 0 0}
\pscustom[linewidth=3.54330707,linecolor=curcolor,linestyle=dashed,dash=3.54330709 3.54330709]
{
\newpath
\moveto(457.405385,338.185722)
\lineto(457.405385,174.500007)
}
}
{
\newrgbcolor{curcolor}{0 0 0}
\pscustom[linewidth=3.54330707,linecolor=curcolor,linestyle=dashed,dash=3.54330709 3.54330709]
{
\newpath
\moveto(447.909905,307.925272)
\lineto(447.909905,174.500007)
}
}
{
\newrgbcolor{curcolor}{0 0 0}
\pscustom[linewidth=3.54330707,linecolor=curcolor,linestyle=dashed,dash=3.54330709 3.54330709]
{
\newpath
\moveto(151.99319,174.500007)
\lineto(151.99319,25.911977)
}
}
{
\newrgbcolor{curcolor}{0 0 0}
\pscustom[linewidth=3.54330707,linecolor=curcolor,linestyle=dashed,dash=3.54330709 3.54330709]
{
\newpath
\moveto(352.20154,174.500007)
\lineto(352.20154,25.072627)
}
}
{
\newrgbcolor{curcolor}{0 0 0}
\pscustom[linewidth=3.54330707,linecolor=curcolor,linestyle=dashed,dash=3.54330709 3.54330709]
{
\newpath
\moveto(219.15886,173.886917)
\lineto(219.15886,44.363107)
}
}
{
\newrgbcolor{curcolor}{0 0 0}
\pscustom[linewidth=7.08661413,linecolor=curcolor]
{
\newpath
\moveto(190.70155,172.500007)
\lineto(213.20155,322.500002)
\lineto(218.20155,43.750007)
\lineto(268.20155,1.250007)
\lineto(300.70155,63.750007)
\lineto(316.95154,38.750007)
\lineto(351.95154,25.000007)
\lineto(378.20154,355.000002)
\lineto(448.201545,308.750002)
\lineto(451.951545,292.500002)
\lineto(458.201545,338.750002)
\lineto(485.701545,322.500002)
\moveto(53.95155,73.750007)
\lineto(136.45155,263.750007)
\lineto(151.45155,25.000007)
\lineto(176.45155,362.500002)
\lineto(190.20155,172.500007)
}
}
{
\newrgbcolor{curcolor}{0 0 0}
\pscustom[linewidth=5.4653163,linecolor=curcolor]
{
\newpath
\moveto(201.189745,173.75000448)
\curveto(201.189745,167.5732642)(196.456737,162.5602576)(190.624995,162.5602576)
\curveto(184.793253,162.5602576)(180.060245,167.5732642)(180.060245,173.75000448)
\curveto(180.060245,179.92674475)(184.793253,184.93975135)(190.624995,184.93975135)
\curveto(196.10316693,184.93975135)(200.62094549,180.59001199)(201.14187151,174.81405108)
}
}
{
\newrgbcolor{curcolor}{0 0 0}
\pscustom[linewidth=3.26097465,linecolor=curcolor]
{
\newpath
\moveto(602.5484554,173.72835)
\lineto(1251.9083204,173.72835)
\lineto(1251.9083204,173.72835)
}
}
{
\newrgbcolor{curcolor}{0 0 0}
\pscustom[linewidth=3.26097465,linecolor=curcolor]
{
\newpath
\moveto(1251.31799325,173.228353)
\lineto(1221.80163575,185.728353)
\lineto(1221.80163575,160.728353)
\lineto(1251.31799325,173.228353)
\closepath
}
}
{
\newrgbcolor{curcolor}{0 0 0}
\pscustom[linewidth=3.54330707,linecolor=curcolor,linestyle=dashed,dash=3.54330709 3.54330709]
{
\newpath
\moveto(766.79135,362.120125)
\lineto(766.79135,175.87988)
}
}
{
\newrgbcolor{curcolor}{0 0 0}
\pscustom[linewidth=3.54330707,linecolor=curcolor,linestyle=dashed,dash=3.54330709 3.54330709]
{
\newpath
\moveto(939.71037,322.369362)
\lineto(939.71037,173.63064)
}
}
{
\newrgbcolor{curcolor}{0 0 0}
\pscustom[linewidth=3.54330707,linecolor=curcolor,linestyle=dashed,dash=3.54330709 3.54330709]
{
\newpath
\moveto(726.71037,263.878467)
\lineto(726.71037,176.25141)
}
}
{
\newrgbcolor{curcolor}{0 0 0}
\pscustom[linewidth=3.54330707,linecolor=curcolor,linestyle=dashed,dash=3.54330709 3.54330709]
{
\newpath
\moveto(1104.7943,354.422858)
\lineto(1104.7943,175.57237)
}
}
{
\newrgbcolor{curcolor}{0 0 0}
\pscustom[linewidth=3.54330707,linecolor=curcolor,linestyle=dashed,dash=3.54330709 3.54330709]
{
\newpath
\moveto(996.54617,175)
\lineto(996.54617,2.30268)
}
}
{
\newrgbcolor{curcolor}{0 0 0}
\pscustom[linewidth=3.54330707,linecolor=curcolor,linestyle=dashed,dash=3.54330709 3.54330709]
{
\newpath
\moveto(1028,175)
\lineto(1028,63.81592)
}
}
{
\newrgbcolor{curcolor}{0 0 0}
\pscustom[linewidth=3.54330707,linecolor=curcolor,linestyle=dashed,dash=3.54330709 3.54330709]
{
\newpath
\moveto(1045.2962,174.56833)
\lineto(1045.2962,39.68167)
}
}
{
\newrgbcolor{curcolor}{0 0 0}
\pscustom[linewidth=3.54330707,linecolor=curcolor,linestyle=dashed,dash=3.54330709 3.54330709]
{
\newpath
\moveto(1179.25,293.667606)
\lineto(1179.25,175.37856)
}
}
{
\newrgbcolor{curcolor}{0 0 0}
\pscustom[linewidth=3.54330707,linecolor=curcolor,linestyle=dashed,dash=3.54330709 3.54330709]
{
\newpath
\moveto(1184.4538,338.685713)
\lineto(1184.4538,175)
}
}
{
\newrgbcolor{curcolor}{0 0 0}
\pscustom[linewidth=3.54330707,linecolor=curcolor,linestyle=dashed,dash=3.54330709 3.54330709]
{
\newpath
\moveto(1174.9584,308.425269)
\lineto(1174.9584,175)
}
}
{
\newrgbcolor{curcolor}{0 0 0}
\pscustom[linewidth=3.54330707,linecolor=curcolor,linestyle=dashed,dash=3.54330709 3.54330709]
{
\newpath
\moveto(743.04165,175)
\lineto(743.04165,26.41197)
}
}
{
\newrgbcolor{curcolor}{0 0 0}
\pscustom[linewidth=3.54330707,linecolor=curcolor,linestyle=dashed,dash=3.54330709 3.54330709]
{
\newpath
\moveto(1079.25,175)
\lineto(1079.25,25.57262)
}
}
{
\newrgbcolor{curcolor}{0 0 0}
\pscustom[linewidth=3.54330707,linecolor=curcolor,linestyle=dashed,dash=3.54330709 3.54330709]
{
\newpath
\moveto(946.20732,174.38691)
\lineto(946.20732,44.8631)
}
}
{
\newrgbcolor{curcolor}{0 0 0}
\pscustom[linewidth=5.4653163,linecolor=curcolor]
{
\newpath
\moveto(926.814745,173.74999748)
\curveto(926.814745,167.5732572)(922.081737,162.5602506)(916.249995,162.5602506)
\curveto(910.418253,162.5602506)(905.685245,167.5732572)(905.685245,173.74999748)
\curveto(905.685245,179.92673775)(910.418253,184.93974435)(916.249995,184.93974435)
\curveto(921.72816693,184.93974435)(926.24594549,180.59000499)(926.76687151,174.81404408)
}
}
{
\newrgbcolor{curcolor}{0 0 0}
\pscustom[linewidth=5.4653163,linecolor=curcolor]
{
\newpath
\moveto(793.064745,173.74999748)
\curveto(793.064745,167.5732572)(788.331737,162.5602506)(782.499995,162.5602506)
\curveto(776.668253,162.5602506)(771.935245,167.5732572)(771.935245,173.74999748)
\curveto(771.935245,179.92673775)(776.668253,184.93974435)(782.499995,184.93974435)
\curveto(787.97816693,184.93974435)(792.49594549,180.59000499)(793.01687151,174.81404408)
}
}
{
\newrgbcolor{curcolor}{0 0 0}
\pscustom[linewidth=7.08661413,linecolor=curcolor]
{
\newpath
\moveto(645.00001,74.25)
\lineto(727.50001,264.249997)
\lineto(742.50001,25.5)
\lineto(767.50001,363)
\lineto(781.25001,173)
\lineto(918.00001,173)
\lineto(940.50001,323)
\lineto(945.50001,44.25)
\lineto(995.5,1.75)
\lineto(1028,64.25)
\lineto(1044.25,39.25)
\lineto(1079.25,25.5)
\lineto(1105.5,355.5)
\lineto(1175.5,309.25)
\lineto(1179.25,293)
\lineto(1185.5,339.25)
\lineto(1213,323)
}
}
{
\newrgbcolor{curcolor}{0 0 0}
\pscustom[fillstyle=solid,fillcolor=curcolor]
{
\newpath
\moveto(450.0000042,36.98228051)
\lineto(457.73734876,37.96616286)
\curveto(458.17202013,35.45959583)(459.38188459,33.5386846)(461.36694578,32.2034234)
\curveto(463.35197876,30.86814947)(466.12669787,30.20051569)(469.69111141,30.20052005)
\curveto(473.28445889,30.20051569)(475.95050753,30.7920158)(477.68926534,31.97502214)
\curveto(479.42796228,33.15801622)(480.29732597,34.54599171)(480.29735901,36.13895278)
\curveto(480.29732597,37.56791304)(479.52938805,38.69234888)(477.99354293,39.51226369)
\curveto(476.92129698,40.07446794)(474.25524834,40.78895322)(469.99538901,41.65572166)
\curveto(464.25756594,42.82699318)(460.28022707,43.84015673)(458.06336047,44.69521534)
\curveto(455.84647226,45.55023624)(454.16570247,46.73323645)(453.02104604,48.24421951)
\curveto(451.87637809,49.75515778)(451.304047,51.42424223)(451.30405104,53.25147789)
\curveto(451.304047,54.9146785)(451.77495233,56.45492134)(452.71676844,57.87221105)
\curveto(453.65857365,59.2894367)(454.94088509,60.46658047)(456.56370661,61.4036459)
\curveto(457.78080647,62.12980849)(459.43984217,62.74473434)(461.5408187,63.2484253)
\curveto(463.64176666,63.75204145)(465.89486755,64.00386823)(468.30012812,64.00390639)
\curveto(471.92245578,64.00386823)(475.10287794,63.58220479)(477.84140413,62.7389148)
\curveto(480.57986917,61.89555102)(482.60113974,60.75354587)(483.90522192,59.31289592)
\curveto(485.2092308,57.87217902)(486.10757328,55.94541135)(486.60025204,53.53258713)
\lineto(478.94984394,52.6892594)
\curveto(478.60206678,54.61014379)(477.59505384,56.10939158)(475.92880211,57.18700727)
\curveto(474.26249304,58.26456028)(471.90796639,58.80335245)(468.86521509,58.80338541)
\curveto(465.27182357,58.80335245)(462.7072007,58.32312465)(461.17133876,57.36270055)
\curveto(459.635449,56.40221341)(458.86751108,55.27777757)(458.86752268,53.98938965)
\curveto(458.86751108,53.16946036)(459.18627776,52.43154934)(459.8238237,51.77565437)
\curveto(460.4613445,51.09628178)(461.46111274,50.53406386)(462.82313142,50.08899891)
\curveto(463.60554317,49.8547172)(465.90935694,49.31592503)(469.73457964,48.47262077)
\curveto(475.26950598,47.27788506)(479.13092969,46.29986013)(481.31886236,45.53854306)
\curveto(483.50672692,44.7771866)(485.2237202,43.67032006)(486.46984736,42.21794014)
\curveto(487.7158961,40.76552747)(488.33894008,38.96174497)(488.33898116,36.80658723)
\curveto(488.33894008,34.69825906)(487.57824685,32.71292703)(486.05689919,30.85058517)
\curveto(484.53547395,28.9882333)(482.34033064,27.54754987)(479.47146268,26.52853058)
\curveto(476.6025303,25.50950991)(473.35690587,24.99999992)(469.73457964,24.99999907)
\curveto(463.73594773,24.99999992)(459.16454367,26.00730703)(456.02035376,28.02192342)
\curveto(452.87614633,30.03653546)(450.86936515,33.02331817)(450.0000042,36.98228051)
\lineto(450.0000042,36.98228051)
\closepath
}
}
{
\newrgbcolor{curcolor}{0 0 0}
\pscustom[fillstyle=solid,fillcolor=curcolor]
{
\newpath
\moveto(1177.73850576,36.1389539)
\lineto(1185.47585031,37.12283625)
\curveto(1185.91052168,34.61626923)(1187.12038614,32.69535799)(1189.10544734,31.36009679)
\curveto(1191.09048032,30.02482287)(1193.86519942,29.35718908)(1197.42961297,29.35719344)
\curveto(1201.02296044,29.35718908)(1203.68900909,29.94868919)(1205.42776689,31.13169553)
\curveto(1207.16646383,32.31468961)(1208.03582752,33.7026651)(1208.03586056,35.29562618)
\curveto(1208.03582752,36.72458643)(1207.2678896,37.84902227)(1205.73204448,38.66893708)
\curveto(1204.65979854,39.23114133)(1201.99374989,39.94562661)(1197.73389056,40.81239505)
\curveto(1191.99606749,41.98366657)(1188.01872862,42.99683012)(1185.80186202,43.85188873)
\curveto(1183.58497382,44.70690963)(1181.90420402,45.88990984)(1180.75954759,47.40089291)
\curveto(1179.61487964,48.91183117)(1179.04254855,50.58091562)(1179.04255259,52.40815128)
\curveto(1179.04254855,54.07135189)(1179.51345388,55.61159474)(1180.45527,57.02888444)
\curveto(1181.3970752,58.44611009)(1182.67938664,59.62325386)(1184.30220816,60.56031929)
\curveto(1185.51930802,61.28648188)(1187.17834372,61.90140773)(1189.27932025,62.40509869)
\curveto(1191.38026821,62.90871484)(1193.6333691,63.16054162)(1196.03862968,63.16057978)
\curveto(1199.66095733,63.16054162)(1202.84137949,62.73887818)(1205.57990568,61.89558819)
\curveto(1208.31837072,61.05222442)(1210.33964129,59.91021926)(1211.64372347,58.46956931)
\curveto(1212.94773236,57.02885241)(1213.84607483,55.10208474)(1214.33875359,52.68926052)
\lineto(1206.68834549,51.84593279)
\curveto(1206.34056833,53.76681718)(1205.33355539,55.26606497)(1203.66730366,56.34368067)
\curveto(1202.00099459,57.42123367)(1199.64646794,57.96002585)(1196.60371664,57.96005881)
\curveto(1193.01032513,57.96002585)(1190.44570225,57.47979804)(1188.90984031,56.51937394)
\curveto(1187.37395055,55.55888681)(1186.60601263,54.43445096)(1186.60602424,53.14606304)
\curveto(1186.60601263,52.32613376)(1186.92477931,51.58822273)(1187.56232525,50.93232776)
\curveto(1188.19984606,50.25295517)(1189.1996143,49.69073725)(1190.56163297,49.24567231)
\curveto(1191.34404473,49.01139059)(1193.6478585,48.47259842)(1197.47308119,47.62929417)
\curveto(1203.00800753,46.43455845)(1206.86943124,45.45653353)(1209.05736391,44.69521645)
\curveto(1211.24522847,43.93385999)(1212.96222175,42.82699346)(1214.20834891,41.37461353)
\curveto(1215.45439765,39.92220086)(1216.07744163,38.11841836)(1216.07748271,35.96326063)
\curveto(1216.07744163,33.85493246)(1215.3167484,31.86960042)(1213.79540075,30.00725856)
\curveto(1212.2739755,28.14490669)(1210.07883219,26.70422327)(1207.20996423,25.68520397)
\curveto(1204.34103185,24.6661833)(1201.09540742,24.15667331)(1197.47308119,24.15667246)
\curveto(1191.47444928,24.15667331)(1186.90304522,25.16398042)(1183.75885531,27.17859682)
\curveto(1180.61464788,29.19320886)(1178.60786671,32.17999156)(1177.73850576,36.1389539)
\lineto(1177.73850576,36.1389539)
\closepath
}
}
{
\newrgbcolor{curcolor}{0 0 0}
\pscustom[linewidth=5.31496048,linecolor=curcolor]
{
\newpath
\moveto(1180.9706532,70.2212263)
\curveto(1218.4706532,70.2212263)(1218.4706532,70.2212263)(1218.4706532,70.2212263)
}
}
\end{pspicture}

%% file: matrice_b2.tex
        \left(
\begin{smallmatrix}
 \frac{1}{8} & 0 & \frac{1}{8} & 0 & 0 & \frac{1}{2} & 0 & 0 & \frac{1}{8} & \frac{1}{8} & 0 & 0 & 0 & 0 & 0 & 0 & 0 & 0 & 0 & 0 & 0 \cr \frac{1}{16} & 
    \frac{1}{16} & \frac{3}{8} & 0 & 0 & \frac{3}{8} & \frac{1}{16} & 0 & 0 & \frac{1}{16} & 0 & 0 & 0 & 0 & 0 & 0 & 0 & 0 & 0 & 0 & 0 \cr 0 & 0 & \frac{1}
   {4} & 0 & 0 & \frac{1}{4} & \frac{1}{4} & 0 & 0 & \frac{1}{4} & 0 & 0 & 0 & 0 & 0 & 0 & 0 & 0 & 0 & 0 & 0 \cr 0 & 0 & \frac{3}{8} & \frac{1}{16} & 0 & \frac{1}
   {16} & \frac{3}{8} & \frac{1}{16} & 0 & \frac{1}{16} & 0 & 0 & 0 & 0 & 0 & 0 & 0 & 0 & 0 & 0 & 0 \cr \frac{1}{16} & 0 & 0 & 0 & \frac{1}{16} & \frac{3}
   {8} & 0 & 0 & \frac{3}{8} & \frac{1}{16} & 0 & 0 & 0 & \frac{1}{16} & 0 & 0 & 0 & 0 & 0 & 0 & 0 \cr 0 & 0 & \frac{1}{16} & 0 & 0 & \frac{3}{8} & \frac{1}
   {16} & 0 & \frac{1}{16} & \frac{3}{8} & 0 & 0 & 0 & \frac{1}{16} & 0 & 0 & 0 & 0 & 0 & 0 & 0 \cr 0 & 0 & \frac{1}{16} & 0 & 0 & \frac{1}{16} & \frac{3}
   {8} & 0 & 0 & \frac{3}{8} & \frac{1}{16} & 0 & 0 & 0 & \frac{1}{16} & 0 & 0 & 0 & 0 & 0 & 0 \cr 0 & 0 & \frac{1}{8} & 0 & 0 & 0 & \frac{1}{2} & \frac{1}
   {8} & 0 & \frac{1}{8} & \frac{1}{8} & 0 & 0 & 0 & 0 & 0 & 0 & 0 & 0 & 0 & 0 \cr 0 & 0 & 0 & 0 & 0 & \frac{1}{4} & 0 & 0 & \frac{1}{4} & \frac{1}
   {4} & 0 & 0 & 0 & \frac{1}{4} & 0 & 0 & 0 & 0 & 0 & 0 & 0 \cr 0 & 0 & 0 & 0 & 0 & \frac{1}{8} & \frac{1}{8} & 0 & 0 & \frac{1}{2} & 0 & 0 & 0 & \frac{1}{8} & 
    \frac{1}{8} & 0 & 0 & 0 & 0 & 0 & 0 \cr 0 & 0 & 0 & 0 & 0 & 0 & \frac{1}{4} & 0 & 0 & \frac{1}{4} & \frac{1}{4} & 0 & 0 & 0 & \frac{1}
   {4} & 0 & 0 & 0 & 0 & 0 & 0 \cr 0 & 0 & 0 & 0 & 0 & 0 & \frac{3}{8} & \frac{1}{16} & 0 & \frac{1}{16} & \frac{3}{8} & \frac{1}{16} & 0 & 0 & \frac{1}
   {16} & 0 & 0 & 0 & 0 & 0 & 0 \cr 0 & 0 & 0 & 0 & 0 & \frac{1}{16} & 0 & 0 & \frac{3}{8} & \frac{1}{16} & 0 & 0 & \frac{1}{16} & \frac{3}{8} & 0 & \frac{1}
   {16} & 0 & 0 & 0 & 0 & 0 \cr 0 & 0 & 0 & 0 & 0 & \frac{1}{16} & 0 & 0 & \frac{1}{16} & \frac{3}{8} & 0 & 0 & 0 & \frac{3}{8} & \frac{1}{16} & 0 & \frac{1}
   {16} & 0 & 0 & 0 & 0 \cr 0 & 0 & 0 & 0 & 0 & 0 & \frac{1}{16} & 0 & 0 & \frac{3}{8} & \frac{1}{16} & 0 & 0 & \frac{1}{16} & \frac{3}{8} & 0 & \frac{1}
   {16} & 0 & 0 & 0 & 0 \cr 0 & 0 & 0 & 0 & 0 & 0 & 0 & 0 & \frac{1}{8} & \frac{1}{8} & 0 & 0 & 0 & \frac{1}{2} & 0 & \frac{1}{8} & \frac{1}
   {8} & 0 & 0 & 0 & 0 \cr 0 & 0 & 0 & 0 & 0 & 0 & 0 & 0 & 0 & \frac{1}{4} & 0 & 0 & 0 & \frac{1}{4} & \frac{1}{4} & 0 & \frac{1}
   {4} & 0 & 0 & 0 & 0 \cr 0 & 0 & 0 & 0 & 0 & 0 & 0 & 0 & 0 & \frac{1}{8} & \frac{1}{8} & 0 & 0 & 0 & \frac{1}{2} & 0 & \frac{1}{8} & \frac{1}
   {8} & 0 & 0 & 0 \cr 0 & 0 & 0 & 0 & 0 & 0 & \frac{1}{16} & 0 & 0 & \frac{1}{16} & \frac{3}{8} & 0 & 0 & 0 & \frac{3}{8} & 0 & 0 & \frac{1}{16} & \frac{1}
   {16} & 0 & 0 \cr 0 & 0 & 0 & 0 & 0 & 0 & 0 & 0 & 0 & \frac{1}{16} & 0 & 0 & 0 & \frac{3}{8} & \frac{1}{16} & \frac{1}{16} & \frac{3}{8} & 0 & 0 & \frac{1}
   {16} & 0 \cr 0 & 0 & 0 & 0 & 0 & 0 & 0 & 0 & 0 & \frac{1}{16} & 0 & 0 & 0 & \frac{1}{16} & \frac{3}{8} & 0 & \frac{3}{8} & \frac{1}{16} & 0 & 0 & \frac{1}
   {16} \cr  
\end{smallmatrix}
\right)

%% file: matrice_a2.tex
\left(
\begin{smallmatrix} 
\frac{1}{2} & 0 & \frac{1}{2} & 0 & 0 & 0 & 0 \cr 0 & \frac{1}
   {2} & \frac{1}
   {2} & 0 & 0 & 0 & 0 \cr 0 & 0 & 1 & 0 & 0 & 0 & 0 \cr 0 & 0 & \frac{1}
   {2} & \frac{1}{2} & 0 & 0 & 0 \cr 0 & 0 & \frac{1}{2} & 0 & \frac{1}
   {2} & 0 & 0 \cr 0 & 0 & \frac{1}{2} & 0 & 0 & \frac{1}
   {2} & 0 \cr 0 & 0 & \frac{1}{2} & 0 & 0 & 0 & \frac{1}{2} \cr
\end{smallmatrix}
\right)